\documentclass[12pt]{article}
\usepackage{amsmath}
\usepackage{amssymb}
\usepackage{amsthm}
\usepackage{verbatim}
\usepackage{mathrsfs}

\newcommand{\R}[0]{\mathbb R}

\newcommand{\Ds}[0]{\mathcal D}

\newtheorem{Th}{Theorem}[section]
\newtheorem{Lemma}{Lemma}[section]
\newtheorem{Prop}[Lemma]{Proposition}
\newtheorem{Coro}[Th]{Corollary}
\newtheorem{Def}{Definition}[section]

\begin{document}

\title{On the regularity of the solution map of the incompressible Euler equation}
\author{H. Inci}

\maketitle

\begin{abstract}
In this paper we consider the incompressible Euler equation on the Sobolev space $H^s(\R^n)$, $s > n/2+1$, and show that for any $T > 0$ its solution map $u_0 \mapsto u(T)$, mapping the initial value to the value at time $T$, is nowhere locally uniformly continuous and nowhere differentiable.
\end{abstract}

\section{Introduction}\label{section_introduction}

The initial value problem for the incompressible Euler equation in $\R^n$, $n \geq 2$, reads as:
\begin{eqnarray}
\nonumber
\partial_t u + (u \cdot \nabla) u &=& -\nabla p \\
\label{E}
\operatorname{div} u &=& 0 \\
\nonumber
u(0)&=& u_0
\end{eqnarray}
where $u(t,x)=\big(u_1(t,x),\ldots,u_n(t,x)\big)$ is the velocity of the fluid at time $t \in \R$ and position $x \in \R^n$, $u \cdot \nabla = \sum_{k=1}^n u_k \partial_k$ acts componentwise on $u$, $\nabla p$ is the gradient of the pressure $p(t,x)$, $\operatorname{div} u=\sum_{k=1}^n \partial_k u_k$ is the divergence of $u$ and $u_0$ is the value of $u$ at time $t=0$ (with assumption $\operatorname{div} u_0 = 0$). The system \eqref{E} (going back to Euler \cite{euler}) describes a fluid motion without friction. The first equation in \eqref{E} reflects the conservation of momentum. The second equation in \eqref{E} says that the fluid motion is incompressible, i.e. that the volume of any fluid portion remains constant during the flow.\\
The unknowns in \eqref{E} are $u$ and $p$. But one can express $\nabla p$ in terms of $u$ -- see Inci \cite{lagrangian}. Thus the evolution of system \eqref{E} is completely described by $u$. \\ \\
To state the result of this paper we have to introduce some notation. For $s \in \R_{\geq 0}$ we denote by $H^s(\R^n)$ the Hilbert space of real valued functions on $\R^n$ of Sobolev class $s$, by $H^s(\R^n;\R^n)$ the vector fields on $\R^n$ of Sobolev class $s$ and by $H^s_\sigma(\R^n;\R^n) \subseteq H^s(\R^n;\R^n)$ the closed subspace consisting of divergence-free vector fields -- see Adams \cite{adams} or Inci, Topalov, Kappeler \cite{composition} for details on Sobolev spaces. In particular we will often need the fact that for $n \geq 1$, $s > n/2$ and $0 \leq s' \leq s$ multiplication
\begin{equation}\label{multiplication}
 H^s(\R^n) \times H^{s'}(\R^n) \to H^{s'}(\R^n),\quad (f,g) \mapsto f \cdot g
\end{equation}
is a continuous bilinear map.\\
The notion of solution for \eqref{E} we are interested in are solutions which lie in $C^0\big([0,T];H^s(\R^n;\R^n)\big)$ for some $T > 0$ and $s > n/2+1$. This is the space of continuous curves on $[0,T]$ with values in $H^s(\R^n;\R^n)$. To be precise we say that $u,\nabla p \in C^0\big([0,T];H^s(\R^n;\R^n)\big)$ is a solution to \eqref{E} if
\begin{equation}\label{RE}
 u(t) = u_0 + \int_0^t -(u(\tau) \cdot \nabla) u(\tau) - \nabla p(\tau) \;d\tau \quad \forall 0 \leq t \leq T
\end{equation}
and $\operatorname{div} u(t)=0$ for all $0 \leq t \leq T$ holds. As $s-1>n/2$ we know by the Banach algebra property of $H^{s-1}(\R^n)$ that the integrand in \eqref{RE} lies in $C^0\big([0,T];H^{s-1}(\R^n;\R^n)\big)$. Due to the Sobolev imbedding and the fact $s > n/2+1$ the solutions considered here are $C^1$ (in the $x$-variable slightly better than $C^1$) and are thus solutions for which the derivatives appearing in \eqref{E} are classical derivatives. For this kind of solutions we have the following well-posedness result (it is here stated in a form which will be convenient later):
\begin{Th}[Kato \cite{kato}]\label{th_kato}
Let $n \geq 2$, $s > n/2+1$ and $T > 0$. Then there is an open maximal (with respect to inclusion) neighborhood $U_T \subseteq H^s_\sigma(\R^n;\R^n)$ of $0$ such that there is a unique solution $u \in C^0\big([0,T];H^s_\sigma(\R^n;\R^n)$ of \eqref{E} for all $u_0 \in U_T$. Moreover the solution map
\[
 E_T:U_T \to H^s_\sigma(\R^n;\R^n),\quad u_0 \mapsto u(T)
\]
is continuous.
\end{Th}

With this we can state the main results of this paper.

\begin{Th}\label{th_nonuniform}
Let $n \geq 2$, $s > n/2+1$ and $T > 0$. Then the solution map $E_T:U_T \to H^s_\sigma(\R^n;\R^n)$ is nowhere locally uniformly continuous.
\end{Th}

Note that this means that $E_T$ is \emph{not} uniformly continuous on any open non-empty subset of $U_T$.

\begin{Coro}
The solution map $E_T$ is nowhere locally Lipschitz.
\end{Coro}

\begin{Th}\label{th_nondifferentiable}
Let $n \geq 2$, $s > n/2+1$ and $T > 0$. Then the solution map $E_T:U_T \to H^s_\sigma(\R^n;\R^n)$ is nowhere differentiable.
\end{Th}

Theorem \ref{th_nondifferentiable} is not implied by Theorem \ref{th_nonuniform}. Indeed, for a continuous function $f:H \to \R$, $(H,\langle \cdot,\cdot \rangle)$ a Hilbert space, which is nowhere locally uniformly continuous, the function $H \to \R$, $x \mapsto \langle x,x \rangle f(x)$ is still nowhere locally uniformly continuous, but differentiable in $x=0$. 

\emph{Related work}:
The question of the regularity of $E_T$ was raised in Ebin, Marsden \cite{ebin_marsden}. A first answer was given in Himonas, Misiolek \cite{himonas}. Himonas and Misiolek construct a pair of sequences of solutions $(u_k)_{k \geq 1},(\tilde u_k)_{k \geq 1}$ to \eqref{E} with the following property: For all $s > 0$ 
\begin{itemize}
\item[(i)] $\big(u_k(0)\big)_{k \geq 1}$ and $\big(\tilde u_k(0)\big)_{k \geq 1}$ are bounded in $H^s_\sigma(\R^n;\R^n)$ with
\[
 \lim_{k \to \infty} ||u_k(0)-\tilde u_k(0)||_s =0.
\]
\end{itemize}
\noindent
and there is a constant $C_s > 0$ so that
\begin{itemize}
\item[(ii)] for all $0 < t < 1$ 
\[
 \liminf_{k \geq 1} ||u_k(t)-\tilde u_k(t)||_s \geq C_s \sin t.
\]
\end{itemize}
This shows that $E_T$ is not uniformly continuous on some bounded sets. \\
We should also mention the result in Kato \cite{kato_burgers}, for the inviscid Burgers' equation
\begin{equation}\label{burgers_eq}
 \partial_t u + u \partial_x u=0,\; u(0)=u_0, \quad t \in \R,\; x \in \R.
\end{equation}
Kato proves that for no $0 < \alpha \leq 1$ and no $t > 0$ the solution map of equation \eqref{burgers_eq}, $u_0 \mapsto u(t)$, is locally $\alpha$-H\"older continuous in the Sobolev space $H^s(\R), s \geq 2$.\\

This paper is more or less an excerpt from the thesis Inci \cite{thesis}. So omitted proofs or references where they can be found are given in Inci \cite{thesis}.

\section{Lagrangian description}\label{section_lagrangian}

Consider now a fluid motion determined by $u$. If one fixes a fluid particle which at time $t=0$ is located at $x \in \R^n$ and whose position at time $t \geq 0$ we denote by $\varphi(t,x) \in \R^n$, we get the following relation between $u$ and $\varphi$
\[
 \partial_t \varphi(t,x) = u\big(t,\varphi(t,x)),
\]
i.e. $\varphi$ is the flow-map of the vectorfield $u$. The second equation in \eqref{E} translates to the well-known relation $\det(d\varphi) \equiv 1$, where $d\varphi$ is the Jacobian of $\varphi$ -- see Majda, Bertozzi \cite{majda}. In this way we get a description of system \eqref{E} in terms of $\varphi$. The description of \eqref{E} in the $\varphi$-variable is called the Lagrangian description of \eqref{E}, whereas the description in the $u$-variable is called the Eulerian description of \eqref{E}. One advantage of the Lagrangian description of \eqref{E} is that it leads to an ODE formulation of \eqref{E}. This was already used in Lichtenstein \cite{lichtenstein} and Gunter \cite{gunter} to get local well-posedness of \eqref{E}.\\ \\
 The discussion in Section \ref{section_introduction} shows that in this paper the state-space of \eqref{E} in the Eulerian description is $H^s(\R^n;\R^n)$, $s > n/2+1$. The state-space of \eqref{E} in the Lagrangian description is given by 
\[
 \Ds^s(\R^n) = \big\{ \varphi:\R^n \to \R^n \;\big|\; \varphi - \operatorname{id} \in H^s(\R^n;\R^n) \mbox{ and } \det d_x\varphi > 0, \;\forall x \in \R^n\big\}
\]
where $\operatorname{id}:\R^n \to \R^n$ is the identity map. Due to the Sobolev imbedding and the condition $s > n/2+1$ the space of maps $\Ds^s(\R^n)$ consists of $C^1$-diffeomorphisms -- see Palais \cite{palais} -- and can be identified via $\Ds^s(\R^n) - \operatorname{id} \subseteq H^s(\R^n;\R^n)$ with an open subset of $H^s(\R^n;\R^n)$. Thus $\Ds^s(\R^n)$ has naturally a real analytic differential structure (for real analyticity we refer to Whittlesey \cite{analyticity}) with the natural identification of the tangent space 
\[
 T\Ds^s(\R^n) \simeq \Ds^s(\R^n) \times H^s(\R^n;\R^n).
\] 
Moreover it is known that $\Ds^s(\R^n)$ is a topological group under composition and that for $0 \leq s' \leq s$ the composition map
\begin{equation}\label{composition}
 H^{s'}(\R^n) \times \Ds^s(\R^n) \to H^{s'}(\R^n), \quad (f,\varphi) \mapsto f \circ \varphi
\end{equation}
is continuous -- see Cantor \cite{cantor_thesis} and Inci, Topalov, Kappeler \cite{composition}. That $\Ds^s(\R^n)$ is the right choice as configuration space for \eqref{E} in Lagrangian coordinates is justified by the fact that every $u \in C^0\big([0,T];H^s(\R^n;\R^n)\big)$, $s > n/2+1$, integrates uniquely to a $\varphi \in C^1\big([0,T];\Ds^s(\R^n)\big)$ fullfilling
\[
 \partial_t \varphi(t) = u(t) \circ \varphi(t) \quad \mbox{for all } 0 \leq t \leq T 
\]
-- see Fischer, Marsden \cite{fischer} or Inci \cite{thesis} for an alternative proof. \\
It turns out that one can describe system \eqref{E} in Lagrangian coordinates by a map, which we call the exponential map associated to \eqref{E}. More precisely (see Inci \cite{lagrangian} for the proof)

\begin{Prop}\label{prop_exp}
Let $n \geq 2$ and $s > n/2+1$. Then there is an open neighborhood $U \subseteq H^s(\R^n;\R^n)$ of $0$ and a real analytic map (called the exponential map associated to \eqref{E})
\[
 \exp:U \to \Ds^s(\R^n)
\]
with the following property: For $T > 0$ let $u \in C^0\big([0,T];H^s(\R^n;\R^n)\big)$ be a solution to \eqref{E} for some $u_0 \in H^s_\sigma(\R^n;\R^n)$ with the corresponding flow $\varphi \in C^1\big([0,T];\Ds^s(\R^n)\big)$ solving $\partial_t \varphi(t)=u(t) \circ \varphi(t)$ for any $0 \leq t \leq T$. Then we have
\[
 \varphi(t)=\exp(t u_0) \quad \forall 0 \leq t \leq T.
\]
\end{Prop}

\noindent
Note that we have $U \cap H^s_\sigma(\R^n;\R^n) = \left. U_T \right|_{T=1}$.

\section{Vorticity}\label{section_vorticity}

A key ingredient for the proof of Theorem \ref{th_nonuniform} and Theorem \ref{th_nondifferentiable} will be the vorticity -- see Bertozzi, Majda \cite{majda} and Inci \cite{thesis} for missing proofs.

\begin{Def}\label{def_vorticity}
Let $u=(u_1,\ldots,u_n)$ be a $C^1$-vector field on $\R^n$. Then the antisymmetric matrix
\[
 \Omega(u):=(\Omega_{ij})_{1 \leq i,j \leq n} := (\partial_i u_j-\partial_j u_i)_{1 \leq i,j \leq n}
\]
is called the vorticity of $u$.
\end{Def}

One can recover a divergence-free vector field from its vorticity by the Biot-Savart law.

\begin{Lemma}\label{lemma_biot_savart}
Let $u$ be a $C^1$-vector field with $\operatorname{div} u=0$ and compactly supported vorticity $\Omega:=\Omega(u)$. Then we have
\[
 u(x) = \frac{1}{\omega_n} \int_{\R^n} \Omega(y) \cdot \frac{x-y}{|x-y|^n}\;dy
\]
for any $x \in \R^n$. Here integration is done componentwise and $\omega_n$ denotes the surface area of a unit sphere in $\R^n$.
\end{Lemma}

Recall that for $u \in H^s(\R^n;\R^n)$, $s \geq 0$, we use the norm $||\cdot||_s$ given by
\[
 ||u||_s^2 = \int_{\R^n} (1+|\xi|^2)^s \left(|\hat u_1(\xi)|^2 + \ldots + |\hat u_n(\xi)|^2\right) \;d\xi
\]
where $\hat f$ denotes the Fourier transform of a function $f$. In the same way we define the norm of a matrix valued map. One than has

\begin{Lemma}\label{lemma_biot_savart_estimate}
Let $s > n/2+ 1$. Then there is a constant $C > 0$ such that we have
\[
 ||du||_{s-1} \leq C ||\Omega(u)||_{s-1}, \quad \forall u \in H^s_\sigma(\R^n;\R^n)
\]
where $du$ denotes the Jacobian matrix of $u$.
\end{Lemma}

A very important property of the vorticity is the following conservation law (an immediate consequence of the vorticity equation -- see Inci \cite{thesis}):

\begin{Lemma}\label{lemma_conservation_law}
Let $n \geq 2$ and $s > n/2+1$. Let further $u \in C^0\big([0,T];H^s(\R^n;\R^n)\big)$, $T > 0$, be a solution of \eqref{E} with $u(0)=u_0 \in H^s_\sigma(\R^n;\R^n)$. We define
\[
 \Omega(t):=\Omega\big(u(t)\big) \mbox{ and } \varphi(t):=\exp(t u_0), \quad 0 \leq t \leq T.
\] 
Then we have for any $0 \leq t \leq T$
\[
 d\varphi(t)^\top \cdot \Omega(t)\circ \varphi(t) \cdot d \varphi(t) = \Omega(0)
\]
or
\begin{equation}\label{conserved}
 \Omega(t) = R_{\varphi(t)}^{-1}\left((d \varphi(t)^{-1})^\top \cdot \Omega(0) \cdot d \varphi(t)^{-1} \right)
\end{equation}
where $R_\varphi$ denotes the map $f \mapsto f \circ \varphi$.
\end{Lemma}

Note that from \eqref{conserved} we conclude that the support of the vorticity $\Omega(t)$ remains compact if $\Omega(0)$ is compact. We have the following estimate for expressions of the form \eqref{conserved}. 

\begin{Lemma}\label{lemma_uniform_estimate}
Let $s > n/2+1$ and $\varphi_\bullet \in \Ds^s(\R^n)$. Then there is $C > 0$ and a neighborhood $U \subseteq \Ds^s(\R^n)$ of $\varphi_\bullet$ such that 
\[
 \frac{1}{C} ||f||_{s-1} \leq ||R_\varphi^{-1} \left((d\varphi^{-1})^\top  \cdot f \cdot d\varphi^{-1}\right)||_{s-1} \leq C ||f||_{s-1}
\]
for any $f \in H^{s-1}(\R^n;\R^{n \times n})$ and any $\varphi \in U$.
\end{Lemma}

\section{Proof of Theorem \ref{th_nonuniform} and  Theorem \ref{th_nondifferentiable}}\label{section_proofs}

Before we prove the theorems, we have to make some preparation. Throughout this section we assume $n \geq 2$ and $s > n/2+1$.\\
First of all we can reduce the proofs to the case $T=1$. This follows from the scaling property of \eqref{E}. In fact, denoting $\phi = \left. E_T \right|_{T=1}$ we have
\begin{equation}\label{scaling}
 E_T(u_0) = T \phi(T u_0),\quad \forall T > 0.
\end{equation}

So the proof of Theorem \ref{th_nonuniform} reduces to

\begin{Prop}\label{prop_nonuniform}
Denote by $\phi$ the map $\left. E_T \right|_{T=1}$ and by $U$ the domain $\left. U_T \right|_{T=1}$. Then
\[
 \phi:U \to H^s_\sigma(\R^n;\R^n)
\]
is nowhere locally uniformly continuous.
\end{Prop}

\begin{proof}[Proof of Theorem \ref{th_nonuniform}]
Follows from Proposition \ref{prop_nonuniform} and \eqref{scaling}.
\end{proof}

In the sequel we use $C_{c,\sigma}^\infty(\R^n;\R^n)$ for the space of smooth and divergence-free vector fields with compact support, i.e.
\[
 C_{c,\sigma}^\infty(\R^n;\R^n)=\big\{ f \in C^\infty(\R^n;\R^n) \; \big| \; \operatorname{div} f=0 \mbox{ and } \operatorname{supp} f \mbox{ compact} \big\}
\]
where $\operatorname{supp} f$ denotes the support of $f$. Note that $C_{c,\sigma}^\infty(\R^n;\R^n) \subseteq H^s_\sigma(\R^n;\R^n)$ is dense -- see Inci \cite{thesis}. The proof of the following (technical) lemma can be found in Inci \cite{thesis} -- the estimates are based on the representation given by the Biot-Savart law as described in Lemma \ref{lemma_biot_savart}.

\begin{Lemma}\label{lemma_dexp} 
Let $U\equiv \left. U_T \right|_{T=1}$ and $u_0 \in U \cap C^\infty_{c,\sigma}(\R^n;\R^n)$. Consider the restriction of the differential of $\exp$ at $u_0$ to $H^s_\sigma(\R^n;\R^n)$, 
\[
 d_{u_0}\exp:H_\sigma^s(\R^n;\R^n) \to H^s(\R^n;\R^n),\quad v_0 \mapsto \left. \partial_\varepsilon \right|_{\varepsilon=0} \exp(u_0 + \varepsilon v_0).
\]
Then there exists $m > 0$ with the following property: For any $R > 0$ there exists $v \in C^\infty_{\sigma,c}(\R^n;\R^n)$ with $|\big(d_{u_0}\exp(v)\big)(x^\ast)| \geq m$, $||v||_s=1$ and support in the ball $B_1(x^\ast)=\{ x \in \R^n \;|\; |x-x^\ast| < 1 \}$ for some $x^\ast \in \R^n$ with $|x^\ast| \geq R$.
\end{Lemma}

For $f_\bullet$ in $H^s_\sigma(\R^n;\R^n)$ we denote by $B_R(f_\bullet) \subseteq H^s_\sigma(\R^n;\R^n)$ the open ball of radius $R > 0$ with center $f_\bullet$, i.e.
\[
 B_R(f_\bullet) = \{ f \in H_\sigma^s(\R^n;\R^n) \; | \; ||f-f_\bullet||_s < R \}.
\] 

Now we can give the proof of Proposition \ref{prop_nonuniform}, copied from Inci \cite{thesis}.

\begin{proof}[Proof of Proposition \ref{prop_nonuniform}]
It suffices to show that for any $u_0$ in the domain $U \subseteq H^s_\sigma(\R^n;\R^n)$ of $\phi$ there exists $R_\ast > 0$ with $B_{R_\ast}(u_0) \subseteq U$ so that $\phi$ is not uniformly continuous on $B_R(u_0)$ for any $0 < R \leq R_\ast$. As  $s > n/2+1$, $H^s(\R^n;\R^n) \hookrightarrow C^1_0(\R^n;\R^n)$. We denote by $C > 0$ the constant of this imbedding
\begin{equation}\label{sobolev_constant}
||f||_{C^1} \leq C ||f||_s.
\end{equation}
By the continuity of the exponential map (Proposition \ref{prop_exp}), there exists $R_0 > 0$ so that $B_{R_0}(u_0) \subseteq U$ and for any $\varphi,\psi \in \exp\big(B_{R_0}(u_0)\big)$
\[
 ||\varphi - \psi||_s < \frac{1}{C}.
\]
Hence by \eqref{sobolev_constant} there is a constant $L > 0$ so that for any $\varphi,\psi \in \exp\big(B_{R_0}(u_0)\big)$
\begin{equation}\label{less_one}
 |\varphi(x)-\psi(x)| < 1 \quad \mbox{and} \quad |\varphi(x)-\varphi(y)| < L |x-y|, \quad \forall x,y \in \R^n.
\end{equation}
By the smoothness of the exponential map (Proposition \ref{prop_exp}) and Taylor's theorem, for any $v,v+h$ in an arbitrary convex subset $V \subseteq U$,
\[
 \exp(v+h)=\exp(v)+d_v\exp(h) + \frac{1}{2} \int_0^1 (1-t) d^2_{v+th}\exp(h,h)\,dt.
\]
By choosing $0 < R_1 \leq R_0$, smaller if necessary, we can ensure that for some $C_1 > 0$ and $\forall v \in B_{R_1}(u_0), h \in B_{R_1}(0)$

\begin{equation}\label{remainder}
||\exp(v+h)-\exp(v) - d_v \exp(h)||_s \leq C_1 ||h||_s^2.
 \end{equation}
As $v \mapsto d_v \exp$ is continuous we get for some $0 < R_2 \leq R_1$ and $\forall v_1,v_2 \in B_{R_2}(u_0), h \in H^s_\sigma(\R^n;\R^n)$.

\begin{equation}\label{near_differential}
||d_{v_1} \exp(h) - d_{v_2}\exp(h)||_s \leq \frac{m}{4C} ||h||_s,
\end{equation}
where $m > 0$ is the constant in the statement of Lemma \ref{lemma_dexp} and $C > 0$ given by \eqref{sobolev_constant}. Finally by choosing $0 < R_3 \leq R_2$, sufficiently small, Lemma \ref{lemma_uniform_estimate} implies that there exists $C_2 > 0$ so that
\begin{equation}\label{vort_estimate}
 \frac{1}{C_2} ||f||_{s-1} \leq ||R_\varphi^{-1}\left([d\varphi^\top]^{-1} f [d\varphi]^{-1}\right)||_{s-1} \leq C_2 ||f||_{s-1}
\end{equation}
for any $f \in H^{s-1}(\R^n;\R^{n\times n})$ and any $\varphi \in \exp\big(B_{R_3}(u_0)\big)$. Now set $R_\ast=R_3$ and take any $0 < R \leq R_\ast$. By the density of $C^\infty_{\sigma,c}(\R^n;\R^n)$ in $H^s_\sigma(\R^n;\R^n)$, there exists $\bar u_0 \in C^\infty_{\sigma,c}(\R^n;\R^n) \cap B_{R/4}(u_0)$. Let $\varphi_\bullet:=\exp(\bar u_0)$ and introduce $K:=\operatorname{supp} \bar u_0$ and
\[
 K'=\{ y \in \R^n \, | \, \operatorname{dist}\big(y,\varphi_\bullet(K)\big) \leq 1 \}
\]
where $\operatorname{dist}\big(y,\varphi_\bullet(K)\big) = \inf_{x \in K} |y - \varphi_\bullet(x)|$ is the distance of $y$ to the set $\varphi_\bullet(K)$. By \eqref{less_one} we see that $K'$ has the property
\begin{equation}\label{in_k_strich}
 \varphi(K) \subseteq K', \qquad \forall \varphi \in \exp\big(B_R(u_0)\big)
\end{equation}
Note that $\lim_{|x| \to \infty} |\varphi_\bullet(x)|=\infty$. By Lemma \ref{lemma_dexp} we then can choose $x^\ast \in \R^n \setminus K'$ and $v \in C^\infty_{\sigma,c}(\R^n;\R^n)$ with $||v||_s=1$ in such a way that
\begin{equation}\label{v_property}
\operatorname{dist}\big(\varphi_\bullet(x^\ast),K'\big) > L + 1 \quad \mbox{and} \quad |\big(d_{\bar u_0} \exp(v)\big)(x^\ast)| \geq m.
\end{equation}
We set $M:=|\big(d_{\bar u_0}\exp(v)\big)(x^\ast)|$ and define $v_k=\frac{R}{4k}v$, $k \geq 1$. As $||v||_s=1$
\begin{equation}\label{v_k_in_ball}
 ||v_k||_s = \frac{R}{4k} < R/3.
\end{equation}
By the definition of $v_k$ we have $|\big(d_{\bar u_0}\exp(v_k)\big)(x^\ast)|=\delta_k:=M\frac{R}{4k}$. By \eqref{less_one} for any $k \geq 1$ there is 
\begin{equation}\label{rho_k_def}
0 < \rho_k < \min(\delta_k/4,1)=\min(\frac{MR}{16k},1)
\end{equation}
such that
\begin{equation}\label{disjoint_support}
 \varphi\big(B_{\rho_k}(x^\ast)\big) \subseteq B_{\delta_k/4}\big(\varphi(x^\ast)\big) \quad \forall \varphi \in \exp\big(B_R(u_0)\big).
\end{equation}
Now choose for each $k \geq 1$, a $w_k \in C^\infty_{\sigma,c}(\R^n;\R^n)$ with 
\begin{equation}\label{w_k_support}
 \operatorname{supp} w_k \subseteq B_{\rho_k}(x^\ast) \quad \mbox{and} \quad ||w_k||_s=R/4
\end{equation}
and define for $k \geq 1$ the pair of initial values
\[
 u_{0,k}=\bar u_0 + w_k \quad \mbox{and} \quad \tilde u_{0,k}=u_{0,k} + v_k.
\]
By our choices $(u_{0,k})_{k \geq 1},(\tilde u_{0,k})_{k \geq 1} \subseteq B_R(u_0)$ and $||u_{0,k} - \tilde u_{0,k}||_s = ||v_k||_s \to 0$ as $k \to \infty$. Denote the diffeomorphims corresponding to $u_{0,k},\tilde u_{0,k}$ by $\varphi_k,\tilde \varphi_k \in \Ds^s(\R^n)$, 
\[
 \varphi_k = \exp(u_{0,k}) \quad \mbox{and} \quad \tilde \varphi_k =\exp(\tilde u_{0,k})
\]
and the solutions of \eqref{RE} corresponding to the initial values $u_{0,k},\tilde u_{0,k}$ by $u_k,\tilde u_k:[0,1] \to H_\sigma^s(\R^n;\R^n)$. The corresponding vorticities at time $t=0$, $\Omega_{0,k}$ and $\tilde \Omega_{0,k}$, and $t=1$, $\Omega_{1,k}$ and $\tilde \Omega_{1,k}$, are then given by 
\begin{equation}\label{vorticities}
\begin{array}{ccccc}
 \Omega_{0,k} & = &\Omega(u_{0,k})&=&\Omega(\bar u_0)+\Omega(w_k) \\
 \tilde \Omega_{0,k} &= &\Omega_{0,k} + \Omega(v_k)&=&\Omega(\bar u_0)+ \Omega(w_k + v_k)
\end{array}
\end{equation}
and
\[
 \Omega_{1,k}=\Omega\big(u_k(1)\big);\quad \tilde \Omega_{1,k} = \Omega\big(\tilde u_k(1)\big).
\]
Note that we have for some $C' > 0$
\begin{equation}\label{u_omega_ineq}
||\phi(u_{0,k})-\phi(\tilde u_{0,k})||_s = ||u_k(1) - \tilde u_k(1)||_s \geq \frac{1}{C'} ||\Omega_{1,k}-\tilde \Omega_{1,k}||_{s-1}.
\end{equation}
We aim at estimating $||\Omega_{1,k}-\tilde \Omega_{1,k}||_{s-1}$ from below. By the conservation law \eqref{conserved} we have
\begin{equation}\label{vorticity_reexpressed}
 \Omega_{1,k}=R_{\varphi_k}^{-1} \left([d\varphi_k^\top]^{-1} \Omega_{0,k} [d\varphi_k]^{-1}\right),\; 
 \tilde \Omega_{1,k} = R_{\tilde \varphi_k}^{-1} \left([d\tilde \varphi_k^\top]^{-1} \tilde \Omega_{0,k} [d\tilde \varphi_k]^{-1}\right).
\end{equation}
By \eqref{v_property} the distance of $\varphi_\bullet(x^\ast)$ to $K'$ is bigger than $L+1$ and hence by \eqref{less_one} 
\[
\operatorname{dist}\big(\varphi(x^\ast),K') > L \quad  \mbox{for any } \varphi \in \exp\big(B_R(u_0)\big).
\]
On the other hand by \eqref{less_one} and $\rho_k < 1$ one has 
\[
 |\varphi(x^\ast)-\varphi(x)| \leq L |x^\ast - x| \leq L \quad \forall x \in \operatorname{supp} w_k.
\]
Combining the two latter displayed inequalities one concludes that
\begin{equation}\label{support_contained1}
 \varphi\big(\operatorname{supp}(w_k)\big) \cap K' =\emptyset, \quad \forall \varphi \in \exp\big(B_R(u_0)\big).
\end{equation}
As $\operatorname{supp}(w_k + v_k) \subseteq B_1(x^\ast)$ the same argument gives
\begin{equation}\label{support_contained2}
 \varphi\big(\operatorname{supp}(w_k + v_k)\big) \cap K' = \emptyset, \quad \forall \varphi \in \exp\big(B_R(u_0)\big).
\end{equation}
By \eqref{in_k_strich},
\[
 \operatorname{supp} R_{\varphi_k}^{-1}\left((d\varphi_k^\top)^{-1} \Omega(\bar u_0) (d\varphi_k)^\top \right) \subseteq K'
\]
and
\[
 \operatorname{supp} R_{\tilde \varphi_k}^{-1}\big((d \tilde \varphi_k^\top)^{-1} \Omega(\bar u_0) (d\tilde \varphi_k)^\top \big) \subseteq K'.
\]
From \eqref{support_contained1}-\eqref{support_contained2},
\[
  \varphi_k\big(\operatorname{supp} \Omega(w_k)\big) \subseteq \R^n\setminus K'
\quad \mbox{and} \quad 
 \tilde \varphi_k \big(\operatorname{supp} \Omega(w_k+v_k)\big) \subseteq \R^n \setminus K'.
\]
By \eqref{vorticities}-\eqref{vorticity_reexpressed} it then follows that
\begin{multline}\label{from_below}
||\Omega_{1,k}-\tilde \Omega_{1,k}||_{s-1} = \\
||R_{\varphi_k}^{-1}\left((d\varphi_k^\top)^{-1} \Omega(\bar u_0) (d\varphi_k)^{-1}\right)- R_{\tilde \varphi_k}^{-1}\left((d\tilde \varphi_k^\top)^{-1} \Omega(\bar u_0) (d\tilde \varphi_k)^{-1}\right)||_{s-1} \\
+ ||R_{\varphi_k}^{-1} \left((d\varphi_k^\top)^{-1} \Omega(w_k) (d\varphi_k)^{-1}\right) - R_{\tilde \varphi_k}^{-1} \left((d\tilde \varphi_k^\top)^{-1} \Omega(w_k+v_k) (d\tilde \varphi_k)^{-1}\right)||_{s-1} \\
\geq ||R_{\varphi_k}^{-1} \left((d\varphi_k^\top)^{-1} \Omega(w_k) (d\varphi_k)^{-1}\right) - R_{\tilde \varphi_k}^{-1} \left((d\tilde \varphi_k^\top)^{-1} \Omega(w_k+v_k) (d\tilde \varphi_k)^{-1}\right)||_{s-1}
\end{multline}
We claim that, for large $k$,
\begin{equation}\label{show_disjoint_support}
\varphi_k\big(\operatorname{supp}(w_k)\big) \cap \tilde \varphi_k\big(\operatorname{supp}(w_k)\big) = \emptyset.
\end{equation}
Indeed by the Taylor formula 
\[
 \tilde \varphi_k - \varphi_k = \exp(\bar u_0 + w_k + v_k) - \exp(\bar u_0 + w_k) =
 d_{\bar u_0 + w_k}\exp(v_k) + \mathcal R_k
\]
where $\mathcal R_k$ is the remainder term. Thus we can write
\begin{equation}\label{diffeo_difference}
\tilde \varphi_k - \varphi_k = d_{\bar u_0} \exp(v_k) + \left( d_{\bar u_0+ w_k}\exp(v_k) - d_{\bar u_0} \exp(v_k)\right) + \mathcal R_k.
\end{equation}
We want to estimate $\tilde \varphi(x^\ast)-\varphi(x^\ast)$ by estimating the three terms on the right-hand side of the latter identity individually. By the Sobolev imbedding \eqref{sobolev_constant} and \eqref{remainder} we get the following estimate for $\mathcal R_k(x^\ast) \in \R^n$
\[
 |\mathcal R_k(x^\ast)| \leq C ||\mathcal R_k||_s \leq C C_1 ||v_k||_s^2 = C C_1 \frac{R^2}{16k^2}.
\]
For $k$ sufficiently large it then follows that
\[
 |\mathcal R_k(x^\ast)| < \frac{\delta_k}{4}.
\]
Furthermore, using \eqref{sobolev_constant} and \eqref{near_differential}, together with $m \leq M$ (cf \eqref{v_property})
\begin{multline*}
 \big| \big(d_{\bar u_0 + w_k}\exp(v_k)\big)(x^\ast) - \big(d_{\bar u_0}\exp(v_k)\big)(x^\ast) \big| \\
\leq C ||d_{\bar u_0 + w_k}\exp(v_k) - d_{\bar u_0}\exp(v_k)||_s  \leq \frac{m}{4} ||v_k||_s \leq \frac{M R}{16k} = \frac{\delta_k}{4}.
\end{multline*}
Finally, for the first term on the right-hand side of \eqref{diffeo_difference} one has by definition,
\[
 \big|d_{\bar u_0} \exp(v_k)(x^\ast)\big| = \delta_k.
\]
Combining the estimates above, \eqref{diffeo_difference} yields for $k$ large enough
\[
 |\tilde \varphi_k(x^\ast) - \varphi_k(x^\ast)| > \frac{\delta_k}{2}.
\]
By \eqref{disjoint_support} we get for large $k$
\[
 \varphi_k\big(B_{\rho_k}(x^\ast)\big) \cap \tilde \varphi_k\big(B_{\rho_k}(x^\ast)\big) = \emptyset
\]
showing \eqref{show_disjoint_support}. It leads by the triangle inequality to the estimate
\begin{multline}\label{limsup1}
 ||R_{\varphi_k}^{-1} \left([d\varphi_k^\top]^{-1} \Omega(w_k) [d\varphi_k]^{-1} \right) - R_{\tilde \varphi_k}^{-1} \left([d\tilde \varphi_k]^{-1} \Omega(w_k + v_k) [d\tilde \varphi_k]^{-1} \right)||_{s-1} \\
 \geq ||R_{\varphi_k}^{-1} \left([d\varphi_k^\top]^{-1} \Omega(w_k) [d\varphi_k]^{-1}\right)||_{s-1} + ||R_{\tilde \varphi_k}^{-1} \left([d\tilde\varphi_k^\top]^{-1} \Omega(w_k) [d\tilde\varphi_k]^{-1}\right)||_{s-1}\\
 - ||R_{\tilde \varphi_k}^{-1} \left([d\tilde\varphi_k^\top]^{-1} \Omega(v_k) [d\tilde\varphi_k]^{-1}\right)||_{s-1}.
\end{multline}
The latter term we can be estimated using \eqref{vort_estimate} by
\begin{equation}\label{limsup2}
 ||R_{\tilde \varphi_k}^{-1} \left([d\tilde \varphi_k^\top]^{-1} \Omega(v_k) [d\tilde \varphi_k]^{-1}\right)||_{s-1} \leq C_2 ||\Omega(v_k)||_{s-1} \leq C_2 C' ||v_k||_s
\end{equation}
which by \eqref{v_k_in_ball} goes to $0$ for $k \to \infty$. For the first two terms on the right-hand side of the inequality \eqref{limsup1} we have again by \eqref{vort_estimate}
\begin{equation}\label{limsup3}
||R_{\varphi_k}^{-1}\left([d\varphi_k^\top]^{-1} \Omega(w_k) [d\varphi_k]^{-1}\right)||_{s-1} \geq \frac{1}{C_2} ||\Omega(w_k)||_{s-1}
\end{equation}
and
\begin{equation}\label{limsup4}
 ||R_{\tilde \varphi_k}^{-1}\left([d\tilde \varphi_k^\top]^{-1} \Omega(w_k) [d\tilde \varphi_k]^{-1}\right)||_{s-1} \geq \frac{1}{C_2} ||\Omega(w_k)||_{s-1}.
\end{equation}
Combining \eqref{limsup1}-\eqref{limsup4}, the inequality \eqref{from_below} then leads to
\[
 \limsup_{k \geq 1} ||\Omega_{1,k}-\tilde \Omega_{1,k}||_{s-1} \geq \limsup_{k \geq 1} \frac{2}{C_2} ||\Omega(w_k)||_{s-1}.
\]
We will get the result by showing that $\limsup_{k \geq 1} ||\Omega(w_k)||_{s-1}$ is bounded away from $0$. In $H^s(\R^n;\R^n)$ the following norm
\[
 |||f|||_s := ||f||_{L^2} + ||df||_{s-1}
\]
is equivalent to the norm $||\cdot||_s$. In particular there exists $C_3 > 0$ so that for any $f \in H^s(\R^n;\R^n)$
\begin{equation}\label{norms_equiv}
\frac{1}{C_3} ||f||_s \leq |||f|||_s \leq C_3 ||f||_s.
\end{equation}
By \eqref{norms_equiv} we thus get $|||w_k|||_s \geq \frac{1}{C_3} \frac{R}{4}$ for all $k \geq 1$. By \eqref{rho_k_def} and \eqref{w_k_support}
\begin{multline}\label{l2_zero}
 ||w_k||_{L^2} \leq ||w_k||_{L^\infty} \operatorname{vol}\big(B_{\rho_k}(x^\ast)\big) \leq C ||w_k||_s \operatorname{vol}\big(B_{\rho_k}(x^\ast)\big) \\
  \leq C \frac{R}{4} \operatorname{vol}\big(B_1(0)\big) \left(\frac{MR}{16k}\right)^n.
\end{multline}
Hence $||w_k||_{L^2}$ goes to $0$ for $k \to \infty$ implying that
\[
 \limsup_{k \geq 1} ||d w_k||_{s-1} \geq \frac{1}{C_3} \frac{R}{4}.
\]
By Lemma \ref{lemma_biot_savart_estimate} 
\[
 \limsup_{k \geq 1} ||\Omega(w_k)||_{s-1} \geq \limsup_{k \geq 1} \frac{1}{C_4} ||d w_k||_{s-1} \geq \frac{1}{C_3 C_4} \frac{R}{4}
\]
for some constant $C_4 >0$. By \eqref{u_omega_ineq} we then conclude
\begin{equation}\label{indep_of_R}
 \limsup_{k \geq 1} ||\phi(u_{0,k})-\phi(\tilde u_{0,k})||_s \geq \limsup_{k \geq 1} \frac{1}{C'} ||\Omega_{1,k}-\tilde \Omega_{1,k}||_{s-1} \geq \frac{1}{4C_3 C_4} R
\end{equation} 
whereas $||u_{0,k} - \tilde u_{0,k}||_s \to 0$. As $(u_{0,k}),(\tilde u_{0,k})$ are in $B_R(u_0)$ this shows that $\phi$ is not uniformly continuous on $B_R(u_0)$.
\end{proof}

\noindent
Finally we can give the proof of Theorem \ref{th_nondifferentiable}

\begin{proof}[Proof of Theorem \ref{th_nondifferentiable}]
By \eqref{scaling} it suffices to consider the case $T=1$, i.e. to prove that
\[
 \phi:U \to H^s_\sigma(\R^n;\R^n)
\]
is nowhere differentiable. The key ingredient is inequality \eqref{indep_of_R}. Let us reformulate it in a convenient way. Let $w \in U$. Then by the last part of the proof of Proposition \ref{prop_nonuniform} there are $R_\ast,C_\ast > 0$ with $B_{R_\ast}(w) \subseteq U$ satisfying the following property: for any $0 < R \leq R_\ast$ there are sequences $(u_{0,k})_{k \geq 1},(\tilde u_{0,k})_{k \geq 1} \subseteq B_{R}(w)$ with
\begin{equation}\label{zero_seq}
\lim_{k \to \infty} ||u_{0,k}-\tilde u_{0,k}||_s =0
\end{equation}
and 
\begin{equation}\label{below_by_R}
||\phi(u_{0,k})-\phi(\tilde u_{0,k})||_s \geq C_\ast R,\qquad \forall k \geq 1.
\end{equation}
Assume now that $\phi$ is differentiable in $w$. For any $h \in H^s_\sigma(\R^n;\R^n)$ with $w+h \in B_{R_\ast}(w)$
\begin{equation}\label{diff_remainder}
\mathcal R(w,h):= \phi(w+h)-\phi(w)+d_w \phi(h).
\end{equation}
By the definition of differentiability there is $0 < R \leq R_\ast$ with
\begin{equation}\label{remainder_estimate}
||\mathcal R(w,h)||_s \leq \frac{C_\ast}{4} ||h||_s
\end{equation}
for any $h \in H^s_\sigma(\R^n;\R^n)$ with $||h||_s \leq R$. Take sequences $(u_{0,k})_{k \geq 1},(\tilde u_{0,k})_{k \geq 1} \subseteq B_R(w)$ satisfying \eqref{zero_seq}-\eqref{below_by_R}. We then get by \eqref{diff_remainder}
\[
 \phi(u_{0,k})=\phi(w+(u_{0,k}-w)) = \phi(w) + d_w \phi(u_{0,k}-w) + \mathcal R(w,u_{0,k}-w)
\]
and a similar expression for $\phi(\tilde u_{0,k})$. Hence
\[
 \phi(u_{0,k})-\phi(\tilde u_{0,k}) = d_w \phi (u_{0,k}-\tilde u_{0,k}) +\mathcal R(w,u_{0,k}-w) - \mathcal R(w,\tilde u_{0,k}-w).
\]
and thus by \eqref{zero_seq}, $||d_w \phi (u_{0,k}-\tilde u_{0,k})||_s \underset{k \to \infty}{\to} 0$ yielding
\begin{multline*}
 \limsup_{k \geq 1} ||\phi(u_{0,k})-\phi(\tilde u_{0,k})||_s \\
\leq \limsup_{k \geq 1} ||\mathcal R(w,u_{0,k}-w)||_s + \limsup_{k \geq 1} ||\mathcal R(w,\tilde u_{0,k}-w)||_s \leq \frac{C_\ast}{2} R
\end{multline*}
where the last inequality follows from \eqref{remainder_estimate}. This is a contradiction to \eqref{below_by_R}. Hence $\phi$ is not differentiable in $w$. As $w$ was arbitrary the claim follows.
\end{proof}

\bibliographystyle{plain}

\flushleft
\author{ Hasan Inci\\
         Institut f\"ur Mathematik\\        
         Universit\"at Z\"urich\\           
         Winterthurerstrasse 190\\          
         CH-8057 Z\"urich\\     
Schwitzerland\\
        {\it email: } {hasan.inci@math.uzh.ch}
}

\end{document}